\documentclass[12pt,reqno]{amsart}
\usepackage{amsfonts}
\usepackage{amsmath,amssymb,amsthm, mathrsfs,amsopn}
\usepackage{a4wide}
\allowdisplaybreaks
\numberwithin{equation}{section}
 \usepackage{lineno}
\usepackage{graphicx}
\usepackage{epstopdf}
\usepackage{float}
\usepackage{stfloats}
\usepackage{subfigure}
\usepackage{soul}
\usepackage{url}
\usepackage{color}
\usepackage[colorlinks, linkcolor=blue, citecolor=blue]{hyperref}

\newtheorem{theorem}{Theorem}[section]
\newtheorem{proposition}[theorem]{Proposition}

\newtheorem{lemma}[theorem]{Lemma}

\theoremstyle{definition}

\newtheorem{definition}[theorem]{Definition}

\newtheorem{remark}[theorem]{Remark}

\newcommand{\R}{\mathbb{R}}
\newcommand{\ds}{\displaystyle}

\newcommand{\eeq}{\end{equation}}
\newcommand{\beq}{\begin{equation}}
\newcommand{\lab}{\label}

\address[L.~L.~ Cui]{\newline\indent ~School of Mathematics and Statistics \& Hubei Key Laboratory of Mathematical Sciences,
central China Normal University,
Wuhan 430079, P. R. China.}
\email{\href{mailto:leileicuiccnu@163.com}{leileicuiccnu@163.com}}

\address[Q.~H.~He]{\newline\indent ~College of Mathematics and Information Science, Guangxi Center for Mathematical Research,
Guangxi University,
Nanning, 530003, P. R. China.}
\email{\href{mailto:heqihan277@163.com}{heqihan277@gxu.edu.cn}}

\address[Z.~Y.~Lv]{\newline\indent School of Mathematical Sciences, Beijing Normal University,
Beijing 100875,P. R. China.}
\email{\href{mailto:zongyanlv0535@163.com}{zongyanlv0535@163.com}}

\address[X.~X.~Zhong]{\newline\indent South China Research Center for Applied Mathematics and Interdisciplinary Studies,
South China Normal University,
Guangzhou 510631, P. R. China.}
\email{\href{mailto:zhongxuexiu1989@163.com}{zhongxuexiu1989@163.com}}

\begin{document}
\title
[ Normalized solutions for a Kirchhoff type equations]
{Normalized solutions for a Kirchhoff type equations with potential in $\mathbb{R}^3$ }
\maketitle

\begin{center}
\author{Leilei Cui$^{a}$, Qihan He$^{b,*}$, Zongyan Lv$^{c}$ and Xuexiu  Zhong$^d$}

\end{center}

\footnote{*:Corresponding author. The research was supported  by the Natural
Science Foundation of China (Nos.11801581,12061012), Guangdong Basic and Applied Basic Research Foundation (2021A1515010034),Guangzhou Basic and Applied Basic Research Foundation(202102020225), Province Natural Science Fund of Guangdong (2018A030310082).}

\begin{abstract}
In the  present paper, we study the existence of normalized solutions  to the following Kirchhoff type equations
\begin{equation*}
-\left(a+b\int_{\R^3}|\nabla u|^2\right)\Delta u+V(x)u+\lambda u=g(u)~\hbox{in}~\R^3
\end{equation*}
satisfying the normalized constraint $\displaystyle\int_{\R^3}u^2=c$,  where $a,b,c>0$ are prescribed constants, and the nonlinearities $g(u)$ are very general and of mass super-critical. Under some suitable assumptions on $V(x)$ and $g(u)$, we can prove the existence of ground state normalized solutions $(u_c, \lambda_c)\in H^1(\R^3)\times\mathbb{R}$, for any given $c>0$.  Due to the presence of the nonlocal term, the weak limit $u$ of any $(PS)_C$ sequence $\{w_n\}$ may not belong to the corresponding Pohozaev manifold, which is different from the local problem. So we have to overcome some
new difficulties  to gain the compactness of a $(PS)_C$ sequence.

\noindent{\bf Keywords:} Kirchhoff type equations, Normalized solutions, Mass super-critical, Pohozaev manifold

\noindent{\bf Mathematics Subject Classification(2020):} 35A01, 35B09,  35J60, 35Q51
\vskip0.23in
\end{abstract}
\section{\textbf{Introduction}}\label{sec-intro}
In this paper, we consider the existence of the ground state solutions to the following Kirchhoff type equations
\begin{equation}\label{problem}
-\left(a+b\int_{\R^3} |\nabla u|^2\right)\Delta u+V(x)u +\lambda u=g(u)~~\text{in}~~\R^3,
\end{equation}
with the $L^{2}$-mass constraint
\begin{equation}\label{eq:20210816-constraint}
\int_{\R^3}|u|^2\mathrm{d}x=c,
\end{equation}
where $a,b,c>0$ are prescribed constants.

Equation \eqref{problem} stems from the stationary analog of the equation
\begin{equation}\label{sec-1-origin}
  \rho \frac{\partial^2 u}{\partial t^2}-\left( \frac{P_0}{h}+\frac{E}{2L}\int_0^L |\frac{\partial u}{\partial x}|^2\mathrm{d}x\right)\frac{\partial^2 u}{\partial x^2}=0,
\end{equation}
which was presented by G. Kirchhoff in \cite{Ki}, as an extension of the classical D'Alembert's wave equations, describing free vibrations of elastic strings.
After the pioneer work of J.L. Lions \cite{JLLions1978}, where a functional analysis approach was proposed, the Kirchhoff type equations began to call attention of researchers.

In Equation \eqref{problem}, if $\lambda\in\mathbb{R}$ is fixed, then we call \eqref{problem} the {\it fixed frequency problem}.
There are various mathematical skills to find critical points of the corresponding energy functional $J_{\lambda}(u)$, including traditional constrained variational method, fixed point theorem and Lyapunov-Schmidt reduction, where
\begin{equation}\label{sec-1-J_lambda}
J_{\lambda}(u)=\frac{a}{2}\|\nabla u\|_2^2+\frac{b}{4}\|\nabla u\|_2^4 +\frac{1}{2} \int_{\R^3} (\lambda+V(x))u^2\mathrm{d}x-\int_{\R^3}G(u)\mathrm{d}x
\end{equation}
and $G(s)=\int_0^sg(t)\mathrm{d}t$. For example, for the case of $V(x)=0$, in \cite{G} and \cite{Y4}, positive ground state for equation \eqref{problem} is obtained. We refer to \cite{ac,HL,Wu} for nontrivial solutions. For the bounded states, we refer to \cite{H,HLP}. And the nodal solutions can be seen in \cite{DPS}. Additionally, interested readers can refer to \cite{13,14} for the existence and local uniqueness of multi-peak solutions constructed by using reduction method in the case of $\lambda=0$ and $g(u)$ satisfying subcritical growth. Some other results related to  the  multiplicity and   concentration of solutions to   Kirchhoff equations can be seen in \cite{F, hz12, WT}.
\vskip 0.1cm

Nowadays, physicists are more interested in solutions satisfying the normalization condition, that is, satisfying the $L^{2}$-mass constraint \eqref{eq:20210816-constraint}. From such a point of view, the mass $c>0$ is prescribed, while the frequency $\lambda$ is unknown and will appear as a Lagrange multiplier. Hence, we call \eqref{problem}--\eqref{eq:20210816-constraint} {\it fixed mass problem} and the solution $u$ is called a {\it normalized solution}. Normalized solutions of \eqref{problem} can be searched as critical points of $J(u)$ constrained on $S_c$, where
\begin{equation}\label{sec-1-J(u)}
J(u)=\frac{a}{2}\|\nabla u\|_2^2+\frac{b}{4}\|\nabla u\|_2^4 +\frac{1}{2} \int_{\R^3} V(x)u^2\mathrm{d}x-\int_{\R^3}G(u)\mathrm{d}x
\end{equation}
and
\begin{equation*}
S_{c}=\{u\in H^1(\R^3): \int_{\R^3} u^2=c\}.
\end{equation*}

When  $V(x)=0$, problem \eqref{problem}--\eqref{eq:20210816-constraint} is a special case ($N=3$) of the following eigenvalue problem
\begin{equation}\label{1.1}
-\left(a+b\ds\int_{\R^N}|\nabla u|^2\right)\Delta u +\lambda u=g(u)~~\text{in}~~\R^N,
\end{equation}
with prescribed $L^{2}$-mass constraint \eqref{eq:20210816-constraint}. For the nonlinearities $g(u)=|u|^{p-2}u$ and more general space dimensions, problem \eqref{1.1} has been widely studied in a sequence of literatures such as \cite{Y1,Y2,Y3}. Specially, Ye in \cite{Y1} considered the following constrained variation problem
\begin{equation}\label{eq1.2}
E_c:=\inf_{u\in S_c}I(u),
\end{equation}
where
\begin{equation}\label{sec-2-I(u)}
I(u):=\frac{a}{2}\|\nabla u\|_2^2+\frac{b}{4}\|\nabla u\|_2^4- \int_{\R^N} G(u)\mathrm{d}x
\end{equation}
with $N=1,2,3.$ By a scaling technique and applying the concentration-compactness principle, Ye proved that there exists $c_p^*\geq 0$, such that $E_c$ is attained if and only if $c>c_p^*$ with $0<p\leq 2+\frac{4}{N}$, or $c\geq c_p^*$  with $2+\frac{4}{N}<p<2+\frac{8}{N}$. The author also obtained that there is no minimizers for problem (\ref{eq1.2}) if $p\geq 2+\frac{8}{N}$. In particular, for the case of $2+\frac{8}{N}<p<2^*$, $E_c=-\infty$. However, a mountain pass critical point was achieved for the functional $I[u]$ constrained on $S_c$. In \cite{Y2}, Ye studied the case of $p=2+\frac{8}{N}$ and proved that there is a mountain pass critical point for the functional $I[u]$ on $S_c$ if $c>c^*$. Also, if $0<c<c^*$, the existence of minimizers for problem \eqref{eq1.2} was established by adding a new perturbation functional on the functional $I[u]$. After that, in \cite{Y3},  Ye also studied the asymptotic behavior of critical points of $I[u]$ on $S_c$ with $p=2+\frac{8}{N}$. Zeng and Zhang \cite{zx} improved the results of \cite{Y1} by scaling technique and energy estimate. They showed the existence and uniqueness of minimizers of (\ref{eq1.2}) with $0<p<2+\frac{8}{N}$, and the existence and uniqueness of the mountain pass type critical points on the $L^2$ normalized manifold for $2+\frac{8}{N}<p<2^*$ or $p=2+\frac{8}{N}$ and $c>c^*$. We point out that some other existence results of normalized solutions of equation (\ref{1.1}) are considered in \cite{lw,zc} and the references therein.
Recently, by using fiber maps and establishing some mini-max structure, He et al. \cite{HLZX} proved the existence and studied the asymptotic behavior  of ground state normalized solutions to problem \eqref{1.1}--\eqref{eq:20210816-constraint} with general nonlinearity $g(u)$ of mass super-critical for any given $c>0$ in the radial subspace $H^{1}_{\mathrm{rad}}(\mathbb{R}^N)(N=1,2,3)$. It is also worth mentioning that in \cite{JeanZhangZhong2021}, Jeanjean et al. developed a global branch approach to the normalized solutions of the following semi-linear Schrodinger equations
\beq\lab{eq:20220106-e1}
\begin{cases}
-\Delta u+\lambda u=g(u)~\hbox{in}~\R^N, N\geq 1,\\
u>0, \int_{\R^N}u^2 \mathrm{d}x=c.
\end{cases}
\eeq
This new approach does not depend on the geometry of the energy functional. So that, they can handle the nonlinearities $g$ in a unified way, which are either mass subcritical, mass critical or mass supercritical.
In spite of the work \cite{JeanZhangZhong2021}, Zeng et al. \cite{ZZZZ2021} prove the existence, non-existence and multiplicity of positive normalized solutions $(\lambda_c, u_c)\in \R\times H^1(\R^N)$ to \eqref{1.1}-\eqref{eq:20210816-constraint} with nonlinearities $g(s)$ in a unified way and for all dimensions $N\geq 1$. Due to the presence of the non-local term $b\int_{\R^N}|\nabla u|^2\mathrm{d}x \Delta u$,  problem \eqref{1.1} lacks the mountain pass geometry in the higher dimension case $N\geq 5$. The results of \cite{ZZZZ2021} is the first attempt in the normalized solution of Kirchhoff problem in high dimension.

 Additionally, Ding and Zhong \cite{ZD} considered the following semi-linear Sch\"{o}dinger equation with potential and general nonlinear term
\begin{equation}\label{semi-problem}
-\Delta u+V(x)u +\lambda u=g(u)~~\text{in}~~\R^N,
\end{equation}
with the $L^{2}$-mass constraint \eqref{eq:20210816-constraint} and $N\geq 3$. Ding and Zhong \cite{ZD} established the existence result of ground state normalized solutions to problem \eqref{semi-problem}--\eqref{eq:20210816-constraint} in a Pohozaev manifold with the corresponding energy functional satisfying a mini-max structure. What's more, the workable space needed is just $H^{1}(\mathbb{R}^N)$ rather than radial $H^{1}_{\mathrm{rad}}(\mathbb{R}^N)$.
\vskip 0.1cm

Motivated by \cite{HLZX} and \cite{ZD}, we intend to consider the existence of ground state normalized solutions to problem \eqref{problem} with $L^{2}$-mass constraint \eqref{eq:20210816-constraint}. Since $V(x)$ is allowed to be non-radial, we can not work in the radial subspace. So in this work, we can neither apply the similar argument as \cite{ZZZZ2021}, nor benefit from the results of \cite{JeanZhangZhong2021}. We shall apply the constraint variational approach here.
As far as we know, there seems that no literatures of problem \eqref{problem}--\eqref{eq:20210816-constraint} involve the case of mass super-critical with general nonlinearities and $V(x)\leq 0$, especially when $g(u)$ is nonhomogeneous and of mass super-critical. In fact, as defined by \eqref{eq1.2},
$E_c$ is not well-defined, since $J(u)$ defined in \eqref{sec-1-J(u)} is unbounded from below on $S_c$.

We emphasize that the compactness argument here is much more complicated than that in \cite{HLZX}, since we can not benefit from the radial embedding like \cite{HLZX}.
On the other hand, due to the presence of the nonlocal term, the weak limit $u$ of any $(PS)_C$ sequence $\{w_n\}$ may not belong to the corresponding Pohozaev manifold, which is different from the local problem \eqref{semi-problem}. So the compactness argument is also more complicated than that in \cite{ZD}. Hence, we believe that these kind of
mathematical difficulties will make this study  interesting.

\vskip 0.1cm
Before   stating our results, we give some assumptions on $g$ and $V(x)$ as follows:

(G1) $g:\R \to \R$ is continuous and odd.

(G2) There exist constants $\alpha,\beta \in \R^{+}$ with $\frac{14}{3}<\alpha< \beta<6$ such that
\begin{equation}\label{G2}
0<\alpha G(s)\leq g(s)s \leq \beta G(s) \quad \text{for}~ s \in \R\backslash \{0\},
\end{equation}
where $G(s)=\int_0^sg(t)\mathrm{d}t$.

(G3) The function defined by $ \widetilde{G}(s):=\frac{1}{2} g(s)s-G(s) $ is of class $C^{1}$ and satisfies
\begin{equation}\label{G3}
\widetilde{G}'(s)s >\frac{14}{3}\widetilde{G}(s) \quad \text{for}~s \in \R\backslash \{0 \}.
\end{equation}

(V1) $\lim\limits_{|x|\to+\infty}V(x)=\sup\limits_{x\in \R^3} V(x) =0$ and there exist some $\sigma_1 \in [0, \frac{3(\alpha-2)-4}{3(\alpha-2)}a)$  such that
\begin{equation}\label{V1}
\left|\int_{\R^3} V(x)u^2\mathrm{d}x\right|\leq\sigma_1\|\nabla u\|_2^2,~\text{for all}~u\in H^1(\R^3).
\end{equation}

(V2) $\nabla V(x)$ exists for a.e. $x\in \R^3$, put $W(x):=\frac{1}{2}\langle\nabla V(x),x\rangle$. There exists some $\sigma_2\in[0, \min\{\frac{3(\alpha-2)(a-\sigma_1)}{4}-a, \frac{6-\beta}{2\beta}a\}]$ such that
\begin{equation}\label{V2}
\left|\int_{\R^3} W (x)u^2\mathrm{d}x\right|\leq\sigma_2\|\nabla u\|_2^2,~\text{for all}~u\in  H^1(\R^3).
\end{equation}

(V3) $\nabla W(x)$ exists for a.e. $ x\in \R^3$, let
\begin{equation*}
\Upsilon (x):=4W(x)+\langle\nabla W(x),x\rangle.
\end{equation*}
There exists some $\sigma_3 \in [0, 2a) $ such that
\begin{equation}\label{V3}
\int_{\R^3} \Upsilon_{+} (x) u^2\mathrm{d}x\leq\sigma_3\|\nabla u\|_2^2,~\text{for all}~u \in  H^1(\R^3).
\end{equation}

\begin{remark}
\begin{itemize}
\item[(i)]We remark again that the variational framework fails in high dimension case due to the presence of the non-local term $b\int_{\R^N}|\nabla u|^2\mathrm{d}x \Delta u$.
    On the other hand, it seem very hard to find examples $V(x)$ admitting (V1)--(V3) for $N=1,2$. Hence, we only consider $N=3$ in present paper.
\item[(ii)]
By direct computations, we can see that $g(s)=\sum\limits_{j=1}^m|s|^{p_j-2}s, p_j\in (\frac{14}{3}, 6)$ satisfies (G1), (G2) and (G3). When $||V||_\frac{3}{2}, ||W||_\frac{3}{2}$ and $ ||\Upsilon_{+} ||_\frac{3}{2}$ are small enough, by Sobolev inequality, we can see that (V1), (V2) and (V3) hold. For example,
if $S||V||_\frac{3}{2}<\frac{3(\alpha-2)-4}{3(\alpha-2)}a$, then
$$|\int_{\R^3}V(x)u^2~dx|\leq ||V||_\frac{3}{2}||u||_{6}^2\leq \sigma_1||\nabla u||_2^2, $$
where $S$ denotes the best Sobolev constant of $D^{1,2}(\R^3)\hookrightarrow L^{6}(\R^3)$ and  $\sigma_1:=S ||V||_\frac{3}{2}$. Another important application is the Hardy potential $V(x)=-\frac{\mu}{|x|^2}$ with suitable small constant $\mu$.
\end{itemize}
\end{remark}

\begin{definition}\lab{def-ground-state}
For any $c>0$, a solution $(u_c,\lambda_c)\in H^1(\R^3)\times \R$ to \eqref{problem}-- \eqref{eq:20210816-constraint} is called a ground state normalized solution, or least energy normalized solution, if
\begin{equation*}
J(u)=\min\left\{J(v): v\in S_c~\hbox{and it solves \eqref{problem} for some $\lambda\in \R$}\right\}.
\end{equation*}
\end{definition}
\vskip 0.1cm
Our main results are given as follows.

\begin{theorem}\label{th1}
 Assume that $g(s)$ satisfies the conditions (G1)-(G3) and $V(x)$ satisfies the hypotheses (V1)-(V3). Then for any $c>0$, problem \eqref{problem}--\eqref{eq:20210816-constraint} admits a ground state normalized solutions $(u_c,\lambda_c)\in S_c\times\R^+$.
\end{theorem}
\begin{remark}
 For any given mass $c>0$, we shall prove that $\widetilde{m}_c$ (see \eqref{var-m_c}) is achieved at some $u_c\in H^1(\R^3)$, and corresponds some $\lambda_c\in \R$. One can check that a solution $u\in H^1(\R^3)$ to \eqref{problem} satisfying $u\in S_c$ must belong to $\mathcal{P}_c$ (see Lemma \ref{Lemma2.1}). Hence, $u_c$ is a ground state normalized solution to \eqref{problem}.
\end{remark}

Before closing this part, we introduce our idea of the proof of Theorem \ref{th1}.
To prove Theorem \ref{th1}, we shall introduce one more constraint, denoted by $\mathcal{P}_c$ (see \eqref{sec-2-P_{C}-def}). We shall prove the new constraint $\mathcal{P}_c$ is natural (see Lemma \ref{lm3.4}). Then we are devoted to searching for the critical points of $J(u)$ (see \eqref{sec-1-J(u)}) on $\mathcal{P}_c$. We shall prove that the functional $J(u)$ possesses a mini-max structure (see Lemma \ref{lm3.7}), and the mini-max value, denoted by $\widetilde{m}_{c}$, coincides to the infimum of $J(u)$ constrained on $\mathcal{P}_c$. That is,
\begin{equation*}
\widetilde{m}_{c}:=\inf_{u\in S_c}\max_{t>0} J(t\star u)=\inf_{u\in \mathcal{P}_c}J(u),
\end{equation*}
where the fiber map $t\mapsto H^1(\R^3)$ is defined by $\displaystyle(t\star u)(x):=t^{\frac{3}{2}}u(tx)$, which preserves the $L^2$-norm. To get the convergence of the $(PS)_{\widetilde{m}_{c}}$ sequence $\{u_n\}$, we firstly apply a similar method to show that there exist a solution $u$ of $-(a+bA^2)\Delta u+V(x)u+\lambda u=g(u),$ and $l\in N$
nontrivial solutions $\omega^1, \omega^2, \cdots, \omega^l$ of $-(a+bA^2)\Delta u+\lambda u=g(u)$ such that \begin{equation}\label{bu4.51}\widetilde{m}_{c}+\frac{bA^4}{4}=J_A(u)+\sum\limits_{k=1}^lI_A(\omega^k)
\end{equation}
and
\begin{equation}\label{bu4.52}
A^2=||\nabla u||_2^2+\sum\limits_{k=1}^l||\nabla \omega^k||_2^2,
\end{equation}
where $A^2=\lim\limits_{n \to +\infty}||\nabla u_n||_2^2$,  $J_A(u)=\frac{a+bA^2}{2}\int_{\R^3}|\nabla u|^2+\frac{1}{2}\int_{\R^3}(V(x)+\lambda )u^2-\int_{\R^3}G(u)$ and
 $I_A(u)=\frac{a+bA^2}{2}\int_{\R^3}|\nabla u|^2+\frac{1}{2}\int_{\R^3} \lambda u^2-\int_{\R^3}G(u).$ We secondly prove that $J_A(u)\geq  \frac{bA^2}{4}||\nabla u||_2^2$ and $I_A(\omega^k)\geq m(||\omega^k||_2^2)+\frac{bA^2}{4}||\nabla \omega^k||_2^2, k=1, 2, \cdots, l$ which, combining the monotonicity of $m(c)$, \eqref{bu4.51} and \eqref{bu4.52}, implies that $l=0$. So we can see that the $(PS)_{\widetilde{m}_{c}}$ sequence $\{u_n\}$ is convergent in $H^1(\R^3).$ We want to emphasize that our methods allow  us not to ask any  monotonicity assumptions  for $V(x)+\frac{1}{4}<\nabla V(x), x>,$ which is one of our innovations.

\vskip 0.1cm
The paper is organized as follows. In Section \ref{sec-preli}, we give some preliminaries and facts. Specially, we construct the Pohozaev sub-manifold $\mathcal{P}_{c}$ and establish some properties of some type of mini-max structure. Then we prove Theorem \ref{th1} in Section \ref{sec-proof-th-1}.
\vskip 0.1cm

\section{\textbf{ Preliminaries}}\label{sec-preli}

 For convenience, we introduce some notations. Denote $u_{+}=\max(u,0)$ for $u\in \mathbb{R}$. For any $1\leq p\leq\infty$, $L^{p}(\mathbb{R}^3)$ is the space of real-valued Lebesgue measurable functions with finite $L^p$ norms $\|u\|_p<+\infty$. $H^{1}(\mathbb{R}^3)$ is the standard Hilbert space, defined as
\begin{equation*}
H^1(\R^3):= \{u\in L^2(\R^3): \nabla u \in L^2(\R^3)\},
\end{equation*}
and the norm in $H^1(\R^3)$ is denoted by $\|u\|_{H^{1}}$. $H^{-1}(\mathbb{R}^3)$ is the dual space of $H^{1}(\mathbb{R}^3)$. We will use $C$ and $C_j(j\in\mathbb N)$ to denote various positive constants, and $o_{n}(1)$ to mean $o_{n}(1)\to 0$ as $n\to 0$.
\vskip 0.2cm

In this section, we will collect some useful preliminary results for the subsequent proof of Theorem \ref{th1}. Firstly, we recall the well-known Gagliardo-Nirenberg inequality with the best constant (see in \cite{WM}):
Let $p\in [2,6)$, then for any $u\in H^{1}(\mathbb{R}^3)$, we have
\begin{equation}\label{G-N-I}
\|u\|_p^p\leq \frac{p}{2\|Q\|_2^{p-2}}\|\nabla u\|_2^{\frac{3p-6}{2}} \|u\|_2^{\frac{6-p}{2}},
\end{equation}
and, up to translations, $Q$ is the unique positive radial solution (we refer to \cite{Kwong1989} for the uniqueness) of
\begin{equation*}
-\frac{3p-6}{4}\Delta Q+\left(\frac{6-p}{4}\right)Q=|Q|^{p-2}Q~\text{in}~\R^3.
\end{equation*}
\vskip 0.1cm
Secondly, following from the assumptions (G1) and (G2), we deduce that for all $t \in \R$ and $s\geq0$,
\begin{equation*}
 \left\{ \begin{array}{ll}
 s^{\beta} G(t) \leq G(ts) \leq s^{\alpha} G(t), \quad \text{if}~ s \leq 1,&   \\
 s^{\alpha} G(t) \leq G(ts) \leq s^{\beta} G(t), \quad \text{if}~ s \geq 1.&    \\
\end{array} \right.
\end{equation*}
Moreover, there exist some constants $ C_1, C_2>0$ such that for all $s \in \R$,
\begin{equation}\label{1.5}
C_1 \min\{|s|^{\alpha},|s|^{\beta}\}\leq G(s)\leq C_2\max\{|s|^{\alpha},|s|^{\beta}\}\leq C_2 (|s|^{\alpha}+|s|^{\beta} ),
\end{equation}
and
\begin{equation}\label{1.51}
\left(\frac{\alpha}{2}-1\right)G(s)\leq\frac{1}{2}g(s)s-G(s)\leq \left(\frac{\beta}{2}-1\right)G(s)\leq \left(\frac{\beta}{2}-1\right)C_2(|s|^{\alpha}+|s|^{\beta}).
\end{equation}
\vskip 0.1cm
Then we introduce the so-called Pohozaev manifold. First, denote
\begin{equation*}
\mathcal{P}:=\{ u\in H^1 (\R^3): P(u)=0\}~~\text{and}~~\mathcal{P}_{\infty} := \{ u\in H^1 (\R^3):  P_{\infty} (u)=0 \},
\end{equation*}
where
\begin{equation}\label{Pohozaev}
P(u)=a\|\nabla u\|_2^2+b\|\nabla u\|_2^4-\int_{\R^3}W(x)u^2(x)\mathrm{d}x-3\int_{\R^3}\widetilde{G}(u)\mathrm{d}x,
\end{equation}
and
\begin{equation*}
P_{\infty}(u)=a\|\nabla u\|_2^2+b\|\nabla u\|_2^4-3\int_{\R^3}\widetilde{G}(u)\mathrm{d}x.
\end{equation*}
We can also define the Pohozaev sub-manifold as follows:
\begin{equation}\label{sec-2-P_{C}-def}
 \mathcal{P}_c:= S_c \cap \mathcal{P}  \quad \text{and} \quad \mathcal{P}_{\infty,c}:= S_c \cap \mathcal{P}_{\infty}.
\end{equation}
\vskip 0.1cm
Then, as usually,   we introduce the fiber map
\begin{equation*}
u(x)\mapsto(t\star u)(x):= t^{\frac{3}{2}} u(tx) \quad x\in \mathbb{R}^3,
\end{equation*}
for $(t,u)\in \R^{+}\times S_c$. Of course, one can easily check that for any $u\in H^{1}(\mathbb{R}^{3})$,
\begin{equation*}
\|t\star u\|_{2}^{2}=\|u\|_{2}^{2}~\text{and}~\|\nabla(t\star u)\|_{2}^{2}=t^{2}\|\nabla u\|_{2}^{2}.
\end{equation*}
It follows that $t\star u \in S_c$ for any $u\in S_{c}$. Define
\begin{equation*}
\Psi_{u}(t):=J(t\star u) \quad \text{and} \quad \Psi_{\infty,u}(t):=I(t\star u).
\end{equation*}
\vskip 0.1cm
\begin{remark}
By direct calculations, we can verify that for any $u\in H^{1}(\mathbb{R}^3)$,
\begin{equation}\label{sec-2-4}
P(u)=\frac{\mathrm{d}}{\mathrm{d}t}J(t\star u)\big|_{t=1}~\text{and}~P_{\infty}(u)=\frac{\mathrm{d}}{\mathrm{d}t}I(t\star u)\big|_{t=1}.
\end{equation}
\end{remark}
\vskip 0.1cm
Finally, we deduce some necessary lemmas and properties for the proof of Theorem \ref{th1}.
\vskip 0.1cm
\begin{lemma}\label{Lemma2.1}
Suppose that $u\in H^1(\R^3)$ is a weak  solution of \eqref{problem}, then $u\in \mathcal{P}$.
\end{lemma}
\begin{proof} Assume that  $u\in H^1(\R^3)$ is a weak   solution of \eqref{problem}. By the standard regularity theory, we obtain that $u\in C^2(\R^3)$. So we have
\begin{equation}\label{test-u}
a\|\nabla u\|_2^2+b\|\nabla u\|_2^4+\int_{\mathbb{R}^3}(V(x)+\lambda)u^2\mathrm{d}x-\int_{\mathbb{R}^3}g(u)u\mathrm{d}x=0.
\end{equation}
Additionally, invoking by the Pohozaev identity, we also deduce that
\begin{equation}\label{Poho}
(a\|\nabla u\|_2^2+b\|\nabla u\|_2^4)+3\int_{\mathbb{R}^3}(V(x)+\lambda)u^2\mathrm{d}x-6\int_{\mathbb{R}^3} G(u)\mathrm{d}x+\int_{\mathbb{R}^3}\langle\nabla V(x), x\rangle u^2\mathrm{d}x=0.
\end{equation}
Eliminating the parameter $\lambda$ from the above equalities \eqref{test-u}--\eqref{Poho}, we conclude that
\begin{equation*}
a\|\nabla u\|_2^2 +b\|\nabla u\|_2^4 -\int_{\R^3}W(x)u^2\mathrm{d}x -3\int_{\R^3}\widetilde{G}(u)\mathrm{d}x=0.
\end{equation*}
The proof of Lemma \ref{Lemma2.1} is complete.
\end{proof}
\vskip 0.1cm
\begin{proposition}\label{bupro1}

(i) Let $u \in S_c $.  Then, $t\in \R^{+}$ is a critical point of $\Psi_{\infty,u}(t)=I(t\star u)$ if and only if $t \star u \in \mathcal{P}_{\infty,c}$.

(ii) Let $u \in S_c$. Then, $t\in \R^{+}$ is a critical point for $\Psi_{u}(t)=J(t\star u)$ if and only if $t \star u \in \mathcal{P}_{c}$.
\end{proposition}
\begin{proof}

 The conclusion (i) has been proved in \cite{HLZX}. But, for the readers' convenience, we give it a proof here.  By direct calculations, it yields $(\Psi_{\infty,u})'(t)=\frac{1}{t} P_{\infty}(t \star u)$, which implies that $(\Psi_{\infty,u}){^\prime}(t)=0$ is equivalent to $t \star u \in \mathcal{P}_{\infty,c} $. In other words, $t\in\R^{+}$ is a critical point of $\Psi_{\infty,u}(t)=I(t\star u)$ if and only if $t \star u \in \mathcal{P}_{\infty,c}$.

At the same time, a similar method can be used to prove the conclusion (ii). The proof of Proposition \ref{bupro1} is complete.
\end{proof}
\vskip 0.1cm
\begin{lemma}\label{lm2.2}
Under the assumptions (G1)--(G2) and (V1)--(V2), $J\big|_{\mathcal{P}_{c}}$ is coercive, that is,
\begin{equation*}
\lim_{u\in \mathcal{P}_{c}, \|\nabla u\|_2\rightarrow \infty}J(u)=+\infty.
\end{equation*}
\end{lemma}
\begin{proof}Since $u\in \mathcal{P}_c$, we deduce by the assumptions (G2)  and (V2) that
\begin{equation*}
\begin{aligned}
(a+\sigma_2)\|\nabla u\|_2^2+b\|\nabla u\|_2^4&\geq a\|\nabla u\|_2^2+b\|\nabla u\|_2^4 -\int_{\R^3}W(x)u^2\mathrm{d}x\\
 &=3\int_{\R^3}\left(\frac{1}{2}g(u)u-G(u)\right)\mathrm{d}x\\
 &\geq\frac{3(\alpha-2)}{2}\int_{\mathbb{R}^{3}}G(u)\mathrm{d}x,
\end{aligned}
\end{equation*}
which, together with (V1), implies that, as  $\|\nabla u\|_2 \to +\infty$,
\begin{equation}\label{17.6}
\begin{aligned}
J(u)&=\frac{a}{2}\|\nabla u\|_2^2+\frac{b}{4}\|\nabla u\|_2^4+\frac{1}{2}\int_{\R^3}V(x)u^2\mathrm{d}x-\int_{\R^3} G(u)\mathrm{d}x\\
&\geq \frac{1}{2}(a-\sigma_1)\|\nabla u\|_2^2+\frac{b}{4}\|\nabla u\|_2^4-\int_{\R^3}G(u)\mathrm{d}x\\
&\geq\left(\frac{1}{2}(a-\sigma_1)-\frac{2(a+\sigma_2)}{3(\alpha-2)}\right)\|\nabla u\|_2^2+\left( \frac{b}{4}-\frac{2b}{3(\alpha-2)}\right)\|\nabla u\|_2^4 \to +\infty.
\end{aligned}
\end{equation}
Hence,
\begin{equation*}
\lim\limits_{u\in \mathcal{P}_c, \|\nabla u\|_2\to +\infty}J(u)=+\infty.
\end{equation*}
The proof of Lemma \ref{lm2.2} is complete.
\end{proof}
\vskip 0.1cm
\begin{proposition}\label{lm2.4}
For any critical point of $J\big|_{\mathcal{P}_{c}}$, if $(\Psi_{u})''(1)\neq0$, then there exists some $\lambda \in \R$ satisfying
\begin{equation*}
J'(u)+\lambda u=0 \quad \hbox{in} \quad H^{-1}(\mathbb{R}^3).
\end{equation*}
\end{proposition}
\begin{proof}
Let $u$ be a critical point of $J(u)$ restricted to $\mathcal{P}_{c}$, then by the Lagrange multipliers rule there exist $\lambda, \mu\in\R$ such that
\begin{equation}
\begin{array}{ll}\label{1.7}
J'(u)+\lambda u+\mu P^{\prime}(u)=0 \quad \text{in} \quad H^{-1}(\mathbb{R}^3).
\end{array}
\end{equation}
It remains to verify $\mu=0$.

We claim that if $u$ solves \eqref{1.7}, then $u$ satisfies
\begin{equation*}
\frac{\mathrm{d}}{\mathrm{d}t}\big(\Phi(t \star u)\big)\big|_{t=1}=0,
\end{equation*}
where
\begin{equation*}
\Phi(u):=J(u)+\frac{1}{2} \lambda \|u\|_2^2 +\mu P(u).
\end{equation*}
In fact, we observe that
\begin{equation*}
\begin{aligned}
\Phi(u)&=I(u)+\frac{1}{2}\int_{\R^3} V(x)u^2\mathrm{d}x+\frac{1}{2}\lambda\|u\|_2^2+\mu P_{\infty}(u)-\mu\int_{\R^3} W(x)u^2\mathrm{d}x\\
& =\Phi_\infty(u)+\frac{1}{2}\int_{\R^3} V(x)u^2\mathrm{d}x-\mu\int_{\R^3}W(x)u^2\mathrm{d}x,
\end{aligned}
\end{equation*}
where
\begin{equation*}
\Phi_\infty(u):=I(u) +\frac{1}{2}\lambda\|u\|_2^2 +\mu P_\infty(u).
\end{equation*}
After that, it is not difficult to verify that
\begin{equation*}
\begin{aligned}
\frac{\mathrm{d}}{\mathrm{d}t} \big( \Phi(t \star u) \big)\Big|_{t=1}
&=\frac{\mathrm{d}}{\mathrm{d}t}\big(\Phi_\infty(t \star u)\big)\Big|_{t=1}+\frac{\mathrm{d}}{\mathrm{d}t}\left(\frac{1}{2} \int_{\R^3} V(x)t^{3} u^2(tx)\mathrm{d}x\right)\Big|_{t=1}\\
&\quad -\mu\frac{\mathrm{d}}{\mathrm{d}t}\left(\int_{\R^3}W(x)t^{3}u^2(tx)\mathrm{d}x\right)\Big|_{t=1}.
\end{aligned}
\end{equation*}
By direct calculations, it is easy to see that
\begin{equation*}
\begin{aligned}
\frac{\mathrm{d}}{\mathrm{d}t}\big( \Phi_\infty(t \star u) \big)\big|_{t=1}
&=\big(a\|\nabla u\|_2^2+b\|\nabla u\|_2^4-3\int_{\R^3} \widetilde{G}(u)\mathrm{d}x\big)\\
&\quad +\mu \big(2a\|\nabla u\|_2^2+4b\|\nabla u\|_2^4+9\int_{\R^3}\widetilde{G}(u)\mathrm{d}x-\frac{9}{2}\int_{\R^3} \widetilde{G}'(u) u\mathrm{d}x\big),
\end{aligned}
\end{equation*}
\begin{equation*}
\begin{aligned}
\frac{\mathrm{d}}{\mathrm{d}t}\left(\frac{1}{2}  \int_{\R^3} V(x)t^{3} u^2(tx)\mathrm{d}x\right)\big|_{t=1}
&\overset{y=tx}{\mathop =}\,\frac{\mathrm{d}}{\mathrm{d}t}\left(\frac{1}{2}\int_{\R^3}V(\frac{y}{t})u^2(y)\mathrm{d}y  \right)\Big|_{t=1}\\
&=\left(\frac{1}{2}\int_{\R^3}\langle\nabla V(\frac{y}{t}),-\frac{y}{t^2} \rangle u^2(y)\mathrm{d}y\right)\Big|_{t=1}  \\
&=\left(\frac{1}{2}\int_{\R^3}\langle \nabla V(x),-\frac{x}{t}\rangle (t \star u )^2 \mathrm{d}x\right)\Big|_{t=1}\\
&=-\frac{1}{2}\int_{\R^3} \langle \nabla V(x),x\rangle u^2(x)\mathrm{d}x
\end{aligned}
\end{equation*}
and
\begin{equation*}
\begin{aligned}
\frac{\mathrm{d}}{\mathrm{d}t}\left(\int_{\R^3} W(x) t^{3} u^2(tx)\mathrm{d}x\right)\Big|_{t=1}
&\overset{y=tx}{\mathop =}\,\frac{\mathrm{d}}{\mathrm{d}t}\left(\int_{\R^3}W(\frac{y}{t})u^2(y)\mathrm{d}y\right)\Big|_{t=1} \\
&=\left(\int_{\R^3} \langle \nabla W(\frac{y}{t}), -\frac{y}{t^2} \rangle  u^2(y)\mathrm{d}y\right)\Big|_{t=1}\\
&=\left(\int_{\R^3} \langle \nabla W(x),-x \rangle u^2(tx)t^3 \mathrm{d}x\right)\Big|_{t=1}\\
&=-\int_{\R^3}\langle \nabla W(x),x\rangle u^2(x)\mathrm{d}x.
\end{aligned}
\end{equation*}
As a consequence, we have
\begin{equation*}
\begin{aligned}
\frac{\mathrm{d}}{\mathrm{d}t}\big(\Phi(t \star u)\big)\Big|_{t=1}
&=\big(a\|\nabla u\|_2^2+b\|\nabla u\|_2^4-3\int_{\R^3}\widetilde{G}(u)\mathrm{d}x\big)\\
&\quad +\mu\big(2a\|\nabla u\|_2^2+4b\|\nabla u\|_2^4+9\int_{\R^3}\widetilde{G}(u)\mathrm{d}x-\frac{9}{2}\int_{\R^3} \widetilde{G}'(u) u\mathrm{d}x \big)\\
&\quad -\frac{1}{2}\int_{\R^3}\langle \nabla V(x),x \rangle u^2(x)\mathrm{d}x+\mu\int_{\R^3}\langle \nabla W(x),x\rangle  u^2(x)\mathrm{d}x.
\end{aligned}
\end{equation*}

On the other hand, a solution to \eqref{1.7} must satisfy the so-called Pohozaev identity
\begin{equation*}
\begin{aligned}
&\quad a\|\nabla u\|_2^2+b\|\nabla u\|_2^4+\mu\big(2a\|\nabla u\|_2^2+4b\|\nabla u\|_2^4\big)\\
&=-\mu\left(9\int_{\R^3}\widetilde{G}(u)\mathrm{d}x-\frac{9}{2}\int_{\R^3}\widetilde{G}'(u) u\mathrm{d}x\right)+3\int_{\R^3} \widetilde{G}(u)\mathrm{d}x\\
&\quad +\frac{1}{2}\int_{\R^3}\langle \nabla V(x),x \rangle u^2(x)\mathrm{d}x-\mu\int_{\R^3}\langle \nabla W(x),x\rangle  u^2(x)\mathrm{d}x,
\end{aligned}
\end{equation*}
which implies that $\frac{\mathrm{d}}{\mathrm{d}t}\big( \Phi(t \star u) \big)|_{t=1}=0$. This completes the proof of the claim.
\vskip 0.1cm
Now we deduce by direct computations that
\begin{equation*}
\begin{aligned}
\Phi(t \star u)&=J(t \star u) +\frac{1}{2} \lambda \|u\|_2^2+\mu P(t \star u)=\Psi_{u}(t)+\frac{1}{2} \lambda\|u\|_2^2+\mu t(\Psi_{u})^{\prime}(t),
\end{aligned}
\end{equation*}
which implies that
\begin{equation*}
\begin{aligned}
\frac{\mathrm{d}}{\mathrm{d}t}\Phi(t \star u)&=(1+\mu)(\Psi_{u})'(t)+\mu t(\Psi_{u})''(t).
\end{aligned}
\end{equation*}
Since $u\in\mathcal{P}_{c}$, $P(u)=0$, we deduce by \eqref{sec-2-4} that
\begin{equation*}
\begin{aligned}
0&=\frac{\mathrm{d}}{\mathrm{d}t}\big( \Phi(t \star u) \big)\big|_{t=1} \\
&=(1+\mu)(\Psi_{u})'(1) + \mu (\Psi_{u})''(1)\\
&=(1+\mu)P(u)+\mu (\Psi_{u})''(1)=\mu (\Psi_{u})''(1).
\end{aligned}
\end{equation*}
Finally, by the fact that $(\Psi_{u})''(1)\neq 0$, we get $\mu =0$. The proof of Proposition \ref{lm2.4} is complete.
\end{proof}
\vskip 0.1cm
\begin{lemma}\label{lm3.3}
Assume that the assumptions (G1)--(G2) and (V1)--(V2) hold. Then for any $c>0$, there exists some $\bar{\delta}_c>0$ such that
\begin{equation}\label{sec-2-2}
\inf\limits_{u\in\mathcal{P}_{c}}\|\nabla u\|_2\geq\bar{\delta}_c.
\end{equation}
\end{lemma}
\begin{proof}
Since $u\in\mathcal{P}_c$, we have the following Pohozaev identity
\begin{equation}\label{17.1}
a\|\nabla u\|_2^2+b\|\nabla u\|_2^4-\int_{\R^3} W(x)u^2(x)\mathrm{d}x=3\int_{\R^3}\widetilde{G}(u\big)\mathrm{d}x.
\end{equation}
By the assumption (V2), we observe that for any $u\in H^1(\R^3)$,
\begin{equation}\label{17.2}
a\|\nabla u\|_2^2+b\|\nabla u\|_2^4-\int_{\R^3}W(x)u^2(x)\mathrm{d}x\geq(a-\sigma_2)\|\nabla u\|_2^2+b \|\nabla u\|_2^4.
\end{equation}
By the assumption (G2) and the Gagliardo-Nirenberg inequality \eqref{G-N-I}, we deduce that
\begin{equation*}
3 \ds\int_{\R^3}  \widetilde{G} (u\big) \mathrm{d}x\leq C\int_{\mathbb{R}^3}(|u|^\alpha+|u|^\beta)\mathrm{d}x\leq C(\|\nabla u\|_2^{\frac{3(\alpha-2)}{2}}+\|\nabla u\|_2^{\frac{3(\beta-2)}{2}}),
\end{equation*}
which, together with \eqref{17.1} and \eqref{17.2}, implies that
\begin{equation}\label{sec-2-1}
(a-\sigma_2) \|\nabla u\|_2^2+b\|\nabla u\|_2^4\leq C(\|\nabla u\|_2^{\frac{3(\alpha-2)}{2}}+\|\nabla u\|_2^{\frac{3(\beta-2)}{2}}).
\end{equation}
The assumptions (G2)  and (V2) give  $\frac{3(\alpha-2)}{2}, \frac{3(\beta-2)}{2}>2$ and $a-\sigma_2>0$, and then we conclude from \eqref{sec-2-1} that there exists some $\bar{\delta}_c >0$ such that
\begin{equation*}
\|\nabla u\|_2\geq \bar{ \delta}_c.
\end{equation*}
The proof of Lemma \ref{lm3.3} is complete.
\end{proof}
\vskip 0.1cm
To be much better at distinguishing the types of some critical points for $I\big|_{S_c}$ ($I$ is defined in \eqref{sec-2-I(u)}) and $J\big|_{S_c}$ ($J$ is defined in \eqref{sec-1-J(u)}), we decide to decompose $\mathcal{P}_{\infty,c}$ and $\mathcal{P}_{c}$  into the disjoint unions $\mathcal{P}_{\infty,c}=\mathcal{P}_{\infty,c}^{+}  \cup  \mathcal{P}_{\infty,c}^{-} \cup \mathcal{P}_{\infty,c}^{0}$, $\mathcal{P}_{c}=\mathcal{P}_{c}^{+}  \cup  \mathcal{P}_{c}^{-} \cup \mathcal{P}_{c}^{0}$  respectively, where
\begin{equation*}
\begin{aligned}
&\mathcal{P}_{\infty,c}^{+}:= \{ u\in \mathcal{P}_{\infty,c}: (\Psi_{\infty,u})''(1) >0\},   \quad &\mathcal{P}_{c}^{+}:= \{ u\in \mathcal{P}_{c}: (\Psi_{u})''(1) >0\}, \\
&\mathcal{P}_{\infty,c}^{-}:= \{ u\in \mathcal{P}_{\infty,c}: (\Psi_{\infty,u})''(1) <0\},  \quad &\mathcal{P}_{c}^{-}:= \{ u\in \mathcal{P}_{c}: (\Psi_{u})''(1) <0\},\\
&\mathcal{P}_{\infty,c}^{0}:= \{ u\in \mathcal{P}_{\infty,c}: (\Psi_{\infty,u})''(1) =0\}, \quad &\mathcal{P}_{c}^{0}:= \{ u\in \mathcal{P}_{c}: (\Psi_{u})''(1) =0\}.
\end{aligned}
\end{equation*}
\vskip 0.1cm
\begin{lemma}\label{lm3.4}
Assume that the assumptions (G1)--(G3) and (V1)--(V3) hold. Then $\mathcal{P}_{c}^{-}=\mathcal{P}_{c}$ is closed in $H^{1}(\R^3)$ and it is a natural constraint of $J\big|_{S_c}$.
\end{lemma}
\begin{proof}
For any $u \in \mathcal{P}_c$, we have
\begin{equation}\label{3.31}
a\|\nabla u\|_2^2=-b\|\nabla u\|_2^4+\int_{\R^3} W(x) u^2\mathrm{d}x+3\int_{\R^3}\widetilde{G}(u)\mathrm{d}x.
\end{equation}
By virtue of equality  \eqref{3.31}, the assumptions (V3), (G3) and Lemma \ref{lm3.3}, we obtain that
\begin{equation}\label{bu12.2}
\begin{aligned}
(\Psi_u)''(1)&=a\|\nabla u\|_2^2+3b\|\nabla u\|_2^4+\int_{\R^3} W(x) u^2\mathrm{d}x+\int_{\R^3}\langle\nabla W(x),x \rangle u^2\mathrm{d}x\\
&\quad +12\int_{\R^3}\widetilde{G}(u)\mathrm{d}x-\frac{9}{2}\int_{\R^3}\widetilde{G}'(u)u\mathrm{d}x \\
&=2b\|\nabla u\|_2^4+2\int_{\R^3} W(x) u^2\mathrm{d}x+\int_{\R^3} \langle \nabla W(x),x\rangle u^2\mathrm{d}x\\
&\quad +15\int_{\R^3} \widetilde{G}(u)\mathrm{d}x-\frac{9}{2}\int_{\R^3}\widetilde{G}'(u)u\mathrm{d}x\\
&\leq2b\|\nabla u\|_2^4+2\int_{\R^3} W(x) u^2\mathrm{d}x+\int_{\R^3}\langle\nabla W(x),x \rangle u^2\mathrm{d}x-6\int_{\R^3}\widetilde{G}(u)\mathrm{d}x \\
&=-2a\|\nabla u\|_2^2+4\int_{\R^3}W(x) u^2\mathrm{d}x+\int_{\R^3}\langle\nabla W(x),x \rangle u^2\mathrm{d}x\\
&\leq -2a\|\nabla u\|_2^2+\int_{\R^3}\Upsilon_{+}u^2\leq(-2a+\sigma_3)\|\nabla u\|_2^2<0,
\end{aligned}
\end{equation}
which implies that $\mathcal{P}_{c}^{+}=\mathcal{P}_{c}^{0}=\phi$.  Hence, $\mathcal{P}_{c}^{-}=\mathcal{P}_{c}$ is closed in $H^{1}(\R^3)$. By Proposition \ref{lm2.4}, we can obtain  that $\mathcal{P}_{c}$ is a natural constraint of $J\big|_{S_c}$. The proof of Lemma \ref{lm3.4} is complete.
\end{proof}
\vskip 0.1cm
\begin{remark}\label{sec-2-remark-1}
Let $\{w_n\} \subseteq \mathcal{P}^-_c$ be such that  $J(w_n) \to \widetilde{m}(c)$. Therefore, there exist two sequences $\{\lambda_n\}, \{\mu_n\} \subseteq \R$ such that,
as $n \to +\infty$,
$$J^\prime(w_n)+\lambda_nw_n+\mu_n P^\prime(w_n)\to 0 \quad \hbox{in} \quad H^{-1}(\R^3).$$
Using a similar argument as Proposition \ref{lm2.4}, we can get that,
as $n \to +\infty$,
\begin{equation}\label{bu12.3}
\mu_n (\Psi_{w_n})''(1)\to 0.
\end{equation}

By Lemma \ref{lm3.3} and \eqref{bu12.2}, we have
 $$(\Psi_{w_n})^{\prime\prime}(1)\leq (-2a+\sigma_3)\bar{\delta}_c^2<0,$$
which, together with \eqref{bu12.3}, implies that $\mu_n \to 0 $ as $n \to +\infty.$

Hence, if furthermore $\{w_n\}$ is bounded in $H^1(\R^3)$, then we obtain that, as $n \to +\infty$,
$$J^\prime(w_n)+\lambda_nw_n \to 0 \quad \hbox{in} \quad H^{-1}(\R^3).$$
\end{remark}
\vskip 0.1cm
\begin{lemma}\label{lm3.5}
Assume that the assumptions (G1)--(G3) and (V1)--(V3) hold. Then for every $u \in S_{c}$ with $c>0$, there exists a unique $t_u\in \R^{+}$ such that ${t_u} \star u \in \mathcal{P}_{c}$. Moreover, $t_u$ is the unique critical point of the function $\Psi_u(t)$, and satisfies $\Psi_{u}(t_u) =\max\limits_{t>0} J(t \star u)$.
\end{lemma}
\begin{proof} Let $u\in S_{c}$. Since $u\in H^{1}(\mathbb{R}^3)$, we have $\|\nabla u\|_{2}>0$.  By the assumption (V2) and direct computations, we have
\begin{equation*}
\begin{aligned}
(\Psi_{u})'(t)&=at\|\nabla u\|_2^2+bt^3\|\nabla u\|_2^4-\int_{\R^3} W(x) t^{2} u^2(tx)\mathrm{d}x-3\int_{\R^3}  \widetilde{G}(t^{\frac{3}{2}} u(x)\big) \mathrm{d}x t^{-4} \\
&\geq (a-\sigma_2) \|\nabla u\|_2^2 t+b\|\nabla u\|_2^4 t^3-C
\left(t^{\frac{3}{2}\alpha-4}\|u\|_{\alpha}^{\alpha}+t^{\frac{3}{2}\beta-4}\|u\|_{\beta}^{\beta}\right),
\end{aligned}
\end{equation*}
where $\beta>\alpha>\frac{14}{3}$ and $a-\sigma_2>0$. It yields that $(\Psi_{u})'(t) >0 $ for $t>0$ small enough. Therefore, there exists some $t_1>0$ such that $ \Psi_u(t)$ increases in $t\in (0, t_1)$.

On the other hand, according to the assumption (V1), we obtain
\begin{equation*}
\begin{aligned}
\Psi_{u}(t)&\leq\frac{a}{2}t^2\|\nabla u\|_2^2+\frac{b}{4}t^4\|\nabla u\|_2^4+\frac{1}{2}\sigma_1 t^2\|\nabla u\|_2^2 -\int_{\R^3} G(t^{\frac{3}{2}} u(tx))\mathrm{d}x \\
&\leq\frac{a}{2}t^2\|\nabla u\|_2^2+\frac{b}{4} t^4\|\nabla u\|_2^4+\frac{\sigma_1}{2} t^2 \|\nabla u\|_2^2- t^{\frac{3}{2}\alpha-3}\|u\|_{\alpha}^{\alpha}.
\end{aligned}
\end{equation*}
Since $\alpha >\frac{14}{3}$, we can infer that $\lim\limits_{t\to+\infty} \Psi_{u}(t)=-\infty$. Hence, there exists some $t_2>t_1$ such that
\begin{equation*}
\Psi_{u}(t_2) =\max\limits_{t>0} J(t \star u).
\end{equation*}
It is clear that $(\Psi_{u})^{\prime}(t_2)=0$ and $t_2 \star u \in \mathcal{P}_{c} $ by Proposition \ref{bupro1}. We suppose to the contrary that there exists another $t_3>0$ such that $t_3 \star u \in  \mathcal{P}_{c} $. Without loss of generality, we may assmue $t_3>t_2$. Following from Lemma \ref{lm3.4}, we observe that both $t_2$ and $t_3$ are strict local maximum points of $\Psi_{u}(t)$, which implies that there exists some $t_4 \in (t_2,t_3)$ such that
\begin{equation*}
\Psi_{u}(t_4)=\min\limits_{t \in [t_2,t_3]} \Psi_{u}(t).
\end{equation*}
It follows that $(\Psi_{u})^{\prime}(t_3)=0$ and $(\Psi_{u})^{\prime\prime}(t_3) \geq 0$, which allows us to conclude that $t_4 \star u \in \mathcal{P}_{c}^{+}\cup \mathcal{P}_{c}^{0}$, a contradiction to Lemma \ref{lm3.4}. The proof of Lemma \ref{lm3.5} is complete.
\end{proof}
\vskip 0.1cm
\begin{lemma}\label{lm3.7}
There holds the following mini-max structure
\begin{equation}\label{var-m_c}
\begin{aligned}
\widetilde{m}(c):=\inf\limits_{u\in\mathcal{P}_c }J(u)=\inf\limits_{u \in \mathcal{S}_c}\max \limits_{ t>0 } J( t \star u)  >0.
\end{aligned}
\end{equation}
\end{lemma}
\begin{proof}
For any  $u \in \mathcal{P}_c$,  by Lemma \ref{lm3.5}, we have
$$J(u)=\Psi_u(1)=\max \limits_{ t>0 } J( t \star u)\geq \inf\limits_{u \in \mathcal{S}_c } \max \limits_{ t>0 } J( t \star u),$$
which implies that
\begin{equation}\label{17.4}
\inf\limits_{u \in \mathcal{P}_c } J(u) \geq \inf\limits_{u \in \mathcal{S}_c } \max \limits_{ t>0 } J( t \star u).
\end{equation}

On the other hand,  for any $u\in \mathcal{S}_c $, by Lemma \ref{lm3.5} again, we obtain
that there exists $t_u$ such that $t_u\star u\in \mathcal{P}_c$ and $J(t_u\star u)=\max \limits_{ t>0 } J( t \star u)$.
Therefore, $$\inf\limits_{u \in \mathcal{P}_c } J(u) \leq J(t_u\star u)=\max \limits_{ t>0 } J( t \star u),$$
which implies that
\begin{equation}\label{17.5}
\inf\limits_{u \in \mathcal{P}_c } J(u) \leq \inf\limits_{u \in \mathcal{S}_c } \max \limits_{ t>0 } J( t \star u).
\end{equation}

 \eqref{17.4} and \eqref{17.5} imply that
 $$\inf\limits_{u \in \mathcal{P}_c } J(u)  = \inf\limits_{u \in \mathcal{S}_c } \max \limits_{ t>0 } J( t \star u).$$
By \eqref{17.6},   \eqref{sec-2-2} and (V2), we see
\begin{equation*}
J(u)\geq C_1\overline{\delta}_c^2+C_2\overline{\delta}_c^4>0,~~\hbox{for~any}~u\in \mathcal{P}_c.
\end{equation*}
Hence
\begin{equation*}
\widetilde{m}(c)\geq C_1\overline{\delta}_c^2+C_2\overline{\delta}_c^4>0.
\end{equation*}
The proof of Lemma \ref{lm3.7} is complete.
\end{proof}
\vskip 0.1cm
\begin{lemma}\label{lm3.8}For any $c>0$,
there holds
\begin{equation}\label{sec-2-m(c)}
\widetilde{m}(c)<m(c):=\inf\limits_{u\in\mathcal{P}_{\infty,c}}I(u)=\inf\limits_{u\in S_{c}}\max\limits_{t>0}I(t\star u).
\end{equation}
\end{lemma}
\begin{proof}In \cite{HLZX}, the authors have shown the following facts: (1)~$\inf\limits_{u\in\mathcal{P}_{\infty,c}}I(u)=\inf\limits_{u\in S_{c}}\max\limits_{t>0}I(t\star u);$ (2)~$m(c)$ can be attained. Therefore, it is sufficient to show that $\widetilde{m}(c)<m(c)$.
Since $m(c)$ can be attained, we may let $w(x)\in \mathcal{P}_{\infty,c}$ attain $m(c)$. Following from the standard potential theory and maximum principle, we can see that $w(x)>0$ in $\R^3$. By Lemma \ref{lm3.5}, we can see that there exists $t_w>0$ such that $t_w\star w\in \mathcal{P}_c$, which, combining the fact that $V(x)\not \equiv 0$ and $\sup\limits_{x\in \R^3}V(x)=0$, implies that
\begin{equation*}
\begin{aligned}
\widetilde{m}(c)\leq J(t_w\star w)&=I(t_w\star w)+\frac{1}{2}\int_{\mathbb{R}^{3}} V(x)t_w^3w^2(t_wx)\mathrm{d}x\\
&<I(t_w\star w)\leq \max \limits_{ t>0 } I( t \star w)=I(w)=m(c).
\end{aligned}
\end{equation*}
The proof of Lemma \ref{lm3.8} is complete.
\end{proof}
\vskip 0.1cm
\begin{lemma}\label{lm12.1}
$m(c)$ is strictly decreasing with respect to $c\in (0,+\infty).$
\end{lemma}
\begin{proof}
 A similar argument, as the proof of  Theorem 1.1 of  \cite{ZY-AM}, can be used to show this Lemma. So we omit it here.
\end{proof}
\vskip 0.1cm
\begin{lemma}\label{sec-3-JA>}
Suppose that $\{w_n\}\subseteq\mathcal{P}_c$ is a minimizing sequence for $J\big|_{\mathcal{P}_c}$ at a positive level $\widetilde{m}(c)$. Then there exists a $u\in H^{1}(\mathbb{R}^3)$ with $u\neq 0$  such that
$$ w_n\rightharpoonup u ~\hbox{in}~ H^1(\R^3),$$
\begin{equation}\label{sec-3-2}
J_{A}(u)\geq \frac{bA^{2}}{4}\|\nabla u\|_{2}^2,
\end{equation}
and $u$ solves
\begin{equation}\label{sec-3-u-equation}
-\big(a+bA^{2}\big)\Delta u+V(x)u+\lambda u=g(u)~\text{in}~\R^3
\end{equation}
for some $\lambda>0$, where $A^{2}:=\lim\limits_{n\to+\infty}\|\nabla w_{n}\|_{2}^{2}$ (up to a subsequence), and
$J_{A}(u)$ is defined as
\begin{equation}\label{sec-3-JA(u)}
J_{A}(u):=\frac{a}{2}\|\nabla u\|_2^2+\frac{bA^{2}}{2}\|\nabla u\|_2^2 +\frac{1}{2} \int_{\R^3} V(x)u^2\mathrm{d}x-\int_{\R^3}G(u)\mathrm{d}x.
\end{equation}
\end{lemma}
\begin{proof}
Let $\{ w_n \} \subseteq \mathcal{P}_c $ be a minimizing sequence for $J\big|_{\mathcal{P}_c}$ at a positive level $\widetilde{m}(c)$. By the estimate \eqref{17.6} in Lemma \ref{lm2.2} and the fact that $\{w_n\}\subseteq\mathcal{P}_c$, it is easy to see that $\{w_n\}$ is bounded in $H^1(\R^3)$. Going to a subsequence if necessary (still denoted by $\{w_n\}$), there exists a $u\in H^1(\R^3)$ such that
\begin{equation*}
\begin{cases}
w_n \rightharpoonup u &\text{in}~H^1(\R^3),\\
w_n\to u &\text{in}~L^q_{\mathrm{loc}}(\R^3)~1\leq q<6,\\
w_n\to u &\text{a.e.~in}~\R^3.
\end{cases}
\end{equation*}

We claim that $u\neq 0$. In fact, if $u=0$, then the Br$\acute{e}$zis-Lieb Lemma and (V1) lead to
\begin{equation*}
\begin{aligned}
\int_{\R^3} V(x)w_n^2\mathrm{d}x&=\int_{\R^3} V(x)u^2\mathrm{d}x+\int_{\R^3}V(x)(w_n-u)^2\mathrm{d}x+o_n(1)=o_n(1),
\end{aligned}
\end{equation*}
which implies that $\widetilde{m}(c)+o_n(1)=I(w_n)$. Furthermore, we obtain $(\Psi_{\infty,w_n})'(1)=(\Psi_{w_n})'(1)+o_n(1)=o_n(1)$. It follows from the uniqueness  of the critical point of $\Psi_{\infty,w_n}(t)$(See Corollary 3.9, \cite{HLZX}) that there exists $t_n=1+o_{n}(1)$ such that $t_n\star w_n \in \mathcal{P}_{\infty,c}$. Hence
\begin{equation*}
m(c)\leq I(t_n \star w_n)=I(w_n)+o_n(1)=\widetilde{m}(c)+o_n(1),
\end{equation*}
which contradicts to \eqref{sec-2-m(c)}.
Therefore $u \neq 0$. This completes the proof of this claim.

Since $\{w_n\}$ is bounded in $H^1(\R^3)$, one can check that there exists a constant $C>0$ such that
\begin{equation}\label{sec-3-1}
|\lambda_n|\leq C, ~\hbox{where}~
\lambda_n:=-\frac{\langle J'(w_n),w_n\rangle}{c}.
\end{equation}
On the other hand, by the assumptions (G2), (V2) and $P(w_{n})=0$,
\begin{equation*}
\begin{aligned}
c\lambda_n&=-\langle J'(w_n), w_n  \rangle\\
&=\int_{\mathbb{R}^3} g(w_n)w_n\mathrm{d}x-\int_{\mathbb{R}^3} V(x)w_n^2\mathrm{d}x-a\|\nabla w_n\|_2^2-b\|\nabla w_n\|_2^4\\
&\geq \int_{\mathbb{R}^3} g(w_n)w_n\mathrm{d}x-a\|\nabla w_n\|_2^2-b\|\nabla w_n\|_2^4\\
&=\int_{\mathbb{R}^3} g(w_n)w_n\mathrm{d}x-\left(\int_{\mathbb{R}^3} W(x)w^2_n\mathrm{d}x+3\int_{\mathbb{R}^3}\big(\frac{1}{2}g(w_n)w_n-G(w_n)\big)\mathrm{d}x\right)\\
&=3\int_{\mathbb{R}^3} G(w_n)\mathrm{d}x-\frac{1}{2}\int_{\mathbb{R}^3} g(w_n)w_n\mathrm{d}x-\int_{\mathbb{R}^3} W(x)w_n^2\mathrm{d}x\\
&\geq\left(3-\frac{\beta}{2}\right)\int_{\mathbb{R}^3} G(w_n)\mathrm{d}x-\sigma_2 \|\nabla w_n\|_2^2
\end{aligned}
\end{equation*}
and
\begin{equation*}
\begin{aligned}
(a-\sigma_2)\|\nabla w_n\|_2^2+b\|\nabla w_n\|_2^4&\leq a\|\nabla w_n\|_2^2+b\|\nabla w_n\|_2^4-\int_{\mathbb{R}^3} W(x)w^2_n\mathrm{d}x\\
&=3\int_{\mathbb{R}^3}\left(\frac{1}{2}g(w_n)w_n-G(w_n)\right)\mathrm{d}x\\
&\leq \frac{3(\beta-2)}{2}\int_{\mathbb{R}^3} G(w_n)\mathrm{d}x.
\end{aligned}
\end{equation*}
Therefore, by the assumption (V2) and \eqref{sec-2-2}, we have
\begin{equation*}
\begin{aligned}
c\lambda_n&\geq \left(3-\frac{\beta}{2}\right)\cdot\frac{2}{3(\beta-2)}\left((a-\sigma_2) \|\nabla w_n\|_2^2+b\|\nabla w_n\|_2^4\right)-\sigma_2 \|\nabla w_n\|_2^2\\
&=\left(\left(3-\frac{\beta}{2}\right)\cdot\frac{2}{3(\beta-2)}(a-\sigma_2)-\sigma_2\right)\|\nabla w_n\|_2^2\\
&\quad +\left(3-\frac{\beta}{2}\right)\cdot\frac{2b}{3(\beta-2)}\|\nabla w_n\|_2^4\\
&\geq C_1\overline{\delta}_c^2+C_2\overline{\delta}_c^4>0.
\end{aligned}
\end{equation*}
Combining this estimation with \eqref{sec-3-1}, there exist two positive constants $C_2>C_1>0$ such that $C_1\leq \lambda_n\leq C_2.$
Hence, up to a subsequence, we have $\lambda_n \rightarrow \lambda>0$ for some $\lambda>0$.

 On the other hand, since $J^\prime(w_n)+\lambda_n w_n\to 0$ in $H^{-1}(\R^3)$ (see Remark \ref{sec-2-remark-1}), we conclude that $u$ is a nontrivial solution of
\begin{equation*}
-\left(a+bA^{2}\right)\Delta u+V(x)u+\lambda u=g(u) \quad \text{in} \quad \R^3,
\end{equation*}
where $A^{2}:=\lim\limits_{n\to+\infty}\|\nabla w_{n}\|_{2}^{2}\geq \|\nabla u\|_{2}^{2}$ (up to a subsequence). Therefore, \eqref{sec-3-u-equation} is proved.

Since $u$ is a weak solution of  \eqref{sec-3-u-equation}, $u$ satisfies the corresponding Pohozaev identity $P_{A}(u)=0$, where
\begin{equation}\label{Pohozaev-A}
P_{A}(u):=a\|\nabla u\|_2^2+bA^{2}\|\nabla u\|_2^2-\int_{\R^3}W(x)u^2\mathrm{d}x-3\int_{\R^3}\widetilde{G}(u)\mathrm{d}x.
\end{equation}
 By the assumptions (G2), (V2) and the Pohozaev identity \eqref{Pohozaev-A}, we can deduce by the assumption (V2) that
\begin{equation}\label{sec-3-4}
\begin{aligned}
(a+\sigma_2)\|\nabla u\|_2^2+bA^{2}\|\nabla u\|_2^2
&\geq a\|\nabla u\|_2^2+bA^{2}\|\nabla u\|_2^2-\int_{\R^3}W(x)u^2\mathrm{d}x\\
 &=3\int_{\R^3}\left(\frac{1}{2}g(u)u-G(u)\right)\mathrm{d}x\\
 &\geq\frac{3(\alpha-2)}{2}\int_{\mathbb{R}^3} G(u)\mathrm{d}x.
\end{aligned}
\end{equation}
Hence by the assumptions (V1)--(V2) and \eqref{sec-3-4}, we conclude that
\begin{equation*}
\begin{aligned}
&\quad J_{A}(u)-\frac{bA^{2}}{4}\|\nabla u\|_{2}^2\\
&=\frac{a}{2}\|\nabla u\|_2^2+\frac{bA^{2}}{4}\|\nabla u\|_2^2 +\frac{1}{2} \int_{\R^3} V(x)u^2\mathrm{d}x-\int_{\R^3}G(u)\mathrm{d}x\\
&\geq\left(\frac{1}{2}(a-\sigma_1)-\frac{2(a+\sigma_2)}{3(\alpha-2)}\right)\|\nabla u\|_2^2+\left( \frac{b}{4}-\frac{2b}{3(\alpha-2)}\right)A^{2}\|\nabla u\|_2^2\\
&\geq 0.
\end{aligned}
\end{equation*}
So  \eqref{sec-3-2} holds and  the proof of Lemma \ref{sec-3-JA>} is complete.
\end{proof}
\vskip 0.1cm
\begin{lemma}\label{sec-3-convergent}
Assume that the assumptions (V1)--(V3) and (G1)--(G3) hold. For any $c>0$, let $\{w_n\}\subseteq\mathcal{P}_c$ be a minimizing sequence for $J\big|_{\mathcal{P}_c}$ at a positive level $\widetilde{m}(c)$. Then there exists a subsequence of $\{w_n\}$ (still denoted by $\{w_n\}$) and a $u\in H^{1}(\mathbb{R}^3)$ with $u\neq 0$ satisfying
\begin{equation*}
w_{n}\to u \quad \hbox{in} \quad H^{1}(\mathbb{R}^3).
\end{equation*}
\end{lemma}
\begin{proof}
Suppose, by contradiction, that  $w_{n}\nrightarrow u$ in $H^{1}(\mathbb{R}^3)$. Applying a similar argument of Ye \cite{Y4}, there exists an $l\in\mathbb{R}$, a sequence $\{y_{n}^{k}\}\subseteq\mathbb{R}^3$ with $|y_{n}^{k}|\to+\infty$ as $n\to+\infty$ for each $1\leq k\leq l$ and nontrivial solutions $w^{1},w^{2},\cdots,w^{l}$ of the following problem
\begin{equation}\label{sec-3-5}
-\left(a+bA^{2}\right)\Delta u+\lambda u=g(u)
\end{equation}
such that
\begin{equation}\label{sec-3-11}
\widetilde{m}(c)+\frac{bA^{4}}{4}=J_{A}(u)+\sum\limits_{k=1}^{l}I_{A}(w^{k}),
\end{equation}
\begin{equation*}
\left\|w_{n}-u-\sum\limits_{k=1}^{l}w^{k}(\cdot-y_{n}^{k})\right\|_{H^{1}}\to 0,
\end{equation*}
and
\begin{equation}\label{sec-3-8}
A^{2}=\|\nabla u\|_{2}^{2}+\sum\limits_{k=1}^{l}\|\nabla w^{k}\|_{2}^{2},
\end{equation}
where
\begin{equation}\label{sec-3-IA(u)}
I_{A}(u):=\frac{a}{2}\|\nabla u\|_2^2+\frac{bA^{2}}{2}\|\nabla u\|_2^2-\int_{\R^3}G(u)\mathrm{d}x.
\end{equation}

We claim that
\begin{equation*}
I_{A}(w^{k})\geq m(\|w^{k}\|_{2}^{2})+\frac{bA^{2}}{4}\|\nabla w^{k}\|_{2}^{2},
\end{equation*}
where $1\leq k\leq l$ and $m(c)$ is defined in \eqref{sec-2-m(c)}. In fact, since $w^{k}$ is a weak solution to  \eqref{sec-3-5}, it satisfies the corresponding Pohozaev identity $P_{\infty,A}(w^{k})=0$, where
\begin{equation}\label{Pohozaev-infty-A}
P_{\infty,A}(u):=a\|\nabla u\|_2^2+bA^{2}\|\nabla u\|_2^2-3\int_{\R^3}\widetilde{G}(u)\mathrm{d}x.
\end{equation}
Hence by $P_{\infty,A}(w^{k})=0$, we can rewrite $I_{A}(w^{k})$ as
\begin{equation}\label{sec-3-6}
\begin{aligned}
I_{A}(w^{k})
&=\frac{a}{4}\|\nabla w^{k}\|_2^2+\frac{a}{4}\|\nabla w^{k}\|_2^2+\frac{bA^{2}}{2}\|\nabla w^{k}\|_2^2-\int_{\R^3}G(w^{k})\mathrm{d}x\\
&=\frac{a}{4}\|\nabla w^{k}\|_2^2+\frac{bA^{2}}{4}\|\nabla w^{k}\|_2^2+\frac{7}{4}\int_{\mathbb{R}^3}\left(\frac{3}{14}g(w^{k})w^{k}-G(w^{k})\right)\mathrm{d}x.
\end{aligned}
\end{equation}

Notice that, \eqref{sec-3-8} gives
\begin{equation*}
\begin{aligned}
P_{\infty}(w^{k})&
=a\|\nabla w^{k}\|_2^2+b\|\nabla w^{k}\|_2^4- 3\int_{\R^3} \widetilde{G}(w^{k})\mathrm{d}x\\
&< a\|\nabla w^{k}\|_2^2+bA^{2}\|\nabla w^{k}\|_2^2- 3\int_{\R^3} \widetilde{G}(w^{k})\mathrm{d}x\\
&=P_{\infty,A}(w^{k})=0.
\end{aligned}
\end{equation*}
At the same time, for $0<t<1$ sufficiently small, by \eqref{1.51} we have
\begin{equation*}
\begin{aligned}
\int_{\R^3} \widetilde{G}(t\star w^{k})\mathrm{d}x
&=\int_{\R^3} \left(\frac{1}{2}g(t\star w^{k})(t\star w^{k})-G(t\star w^{k})\right)\mathrm{d}x\\
&\leq \left(\frac{\beta}{2}-1\right) C_2 \int_{\R^3}(|t\star w^{k}|^{\alpha}+|t\star w^{k}|^{\beta})\mathrm{d}x\\
&=\left(\frac{\beta}{2}-1\right) C_2 \int_{\R^3}(t^{\frac{3\alpha}{2}-3}|w^{k}|^{\alpha}+t^{\frac{3\beta}{2}-3}|w^{k}|^{\beta})\mathrm{d}x.
\end{aligned}
\end{equation*}
So, for $0<t<1$ sufficiently small, we have
\begin{equation}\label{sec-3-9}
\begin{aligned}
P_{\infty}(t\star w^{k})&=at^{2}\|\nabla w^{k}\|_2^2+bt^{4}\|\nabla w^{k}\|_2^4- 3\int_{\R^3} \widetilde{G}(t\star w^{k})\mathrm{d}x\\
&\geq at^{2}\|\nabla w^{k}\|_2^2+bt^{4}\|\nabla w^{k}\|_2^4-3 \left(\frac{\beta}{2}-1\right) C_2 \int_{\R^3}(t^{\frac{3\alpha}{2}-3}|w^{k}|^{\alpha}+t^{\frac{3\beta}{2}-3}|w^{k}|^{\beta})\mathrm{d}x\\
&>0,
\end{aligned}
\end{equation}
since $\frac{3\alpha}{2}-3>4$ and $\frac{3\beta}{2}-3>4$. Then there exists a $t_{w^{k}}\in (0,1)$ such that $P_{\infty}(t_{w^{k}}\star w^{k})=0$. Therefore, we deduce from Proposition \ref{bupro1} that $t_{w^{k}}$ is the unique critical point of $\Psi_{\infty,u}(t)=I(t\star w^{k})$ and
\begin{equation*}
I(t_{w^{k}}\star w^{k})=\max\limits_{t>0}I(t\star w^{k}).
\end{equation*}
Hence
\begin{equation}\label{sec-3-7}
\begin{aligned}
I(t_{w^{k}}\star w^{k})
&=\frac{a t_{w^{k}}^{2}}{2}\|\nabla w^{k}\|_2^2+\frac{b t_{w^{k}}^{4}}{4}\|\nabla w^{k}\|_2^4-\int_{\R^3}G(t_{w^{k}}\star w^{k})\mathrm{d}x\\
&=\frac{a t_{w^{k}}^{2}}{4}\|\nabla w^{k}\|_2^2+\frac{a t_{w^{k}}^{2}}{4}\|\nabla w^{k}\|_2^2+\frac{b t_{w^{k}}^{4}}{4}\|\nabla w^{k}\|_2^4-\int_{\R^3}G(t_{w^{k}}\star w^{k})\mathrm{d}x\\
&=\frac{a t_{w^{k}}^{2}}{4}\|\nabla w^{k}\|_2^2+\frac{7}{4}\int_{\mathbb{R}^3}\left(\frac{3}{14}g(t_{w^{k}}\star w^{k})(t_{w^{k}}\star w^{k})-G(t_{w^{k}}\star w^{k})\right)\mathrm{d}x\\
&=\frac{a t_{w^{k}}^{2}}{4}\|\nabla w^{k}\|_2^2+\frac{7}{4}\int_{\mathbb{R}^3}t_{w^{k}}^{-3}\left(\frac{3}{14}g(t_{w^{k}}^\frac{3}{2} w^{k})(t_{w^{k}}^\frac{3}{2} w^{k})-G(t_{w^{k}}^\frac{3}{2} w^{k})\right)\mathrm{d}x.
\end{aligned}
\end{equation}
We note that the assumption (G3) implies  that $s^{-2}F(s):=\frac{3}{14}g(s)s^{-1}-G(s)s^{-2}$ increases   in $(0,+\infty)$ and decreases in $(-\infty, 0)$. Therefore
 $$(t_{w^{k}}^\frac{3}{2} w^{k})^{-2}F(t_{w^{k}}^\frac{3}{2} w^{k})\leq ( w^{k})^{-2}F( w^{k}),~~~ x\in \{x\in \R^3: w^k\neq 0\},$$
 where we have used that $t_{w^k}\in (0,1)$ and which implies that
 \begin{equation}\label{15.1}
 t_{w^{k}}^{-3}F(t_{w^{k}}^\frac{3}{2} w^{k})=(t_{w^{k}}^\frac{3}{2} )^{-2}F(t_{w^{k}}^\frac{3}{2} w^{k})\leq F( w^{k}),~~~ x\in \{x\in \R^3: w^k\neq 0\}.
 \end{equation}
 It is easy to see that  $F(0)=0$. Thus
 \begin{equation}\label{15.2}
 t_{w^{k}}^{-3}F(t_{w^{k}}^\frac{3}{2} w^{k})=0=F( w^{k}),~~~ x\in \{x\in \R^3: w^k=0\},
 \end{equation}
 which, together with \eqref{15.1} and the definition of $F(s)$, implies that
 \begin{equation}\label{15.3}
 t_{w^{k}}^{-3}\left(\frac{3}{14}g(t_{w^{k}}^\frac{3}{2} w^{k})(t_{w^{k}}^\frac{3}{2} w^{k})-G(t_{w^{k}}^\frac{3}{2} w^{k})\right)\leq \left(\frac{3}{14}g(w^{k})w^{k}-G(w^{k})\right),~~~ x\in \R^3.
 \end{equation}
 So combining with \eqref{sec-3-6}--\eqref{sec-3-7} and \eqref{15.3}, we have
\begin{equation*}
I_{A}(w^{k})\geq I(t_{w^{k}}\star w^{k})+\frac{bA^{2}}{4}\|\nabla w^{k}\|_2^2\geq m(\|w^{k}\|_{2}^{2})+\frac{bA^{2}}{4}\|\nabla w^{k}\|_{2}^{2}.
\end{equation*}
This completes the proof of the  claim.

Now we deduce from \eqref{sec-3-2}, \eqref{sec-3-11}, \eqref{sec-3-8} and Lemmas \ref{lm3.8} and \ref{lm12.1} that
\begin{equation*}
\begin{aligned}
\widetilde{m}(c)+\frac{bA^{4}}{4}
&=J_{A}(u)+\sum\limits_{k=1}^{l}I_{A}(w^{k})\\
&\geq \frac{bA^{2}}{4}\|\nabla u\|_{2}^2+lm(c)+\frac{bA^{2}}{4}\sum\limits_{k=1}^{l}\|\nabla w^{k}\|_{2}^{2}\\
&\geq m(c)+\frac{bA^{4}}{4}\\
&> \widetilde{m}(c)+\frac{bA^{4}}{4},
\end{aligned}
\end{equation*}
which is impossible. So we obtain that $w_n \to u$ in $H^1(\R^3)$.
\end{proof}
\vskip 0.2cm
\section{\bf{Proof of Theorem \ref{th1} } }\label{sec-proof-th-1}
In this section, we prove Theorem \ref{th1}.
\vskip 0.1cm
\begin{proof}[The proof of Theorem \ref{th1}:]  Let $\{w_n\}\subseteq\mathcal{P}_c$ be a minimizing sequence for $\widetilde{m}(c)$. According to Lemma \ref{sec-3-convergent},
 we obtain that there exists a subsequence of $\{w_n\}$ (still denoted by $\{w_n\}$) and a $u\in H^{1}(\mathbb{R}^3)$ with $u\neq 0$  such that
$$w_{n}\to u \quad \text{in} \quad H^{1}(\mathbb{R}^3),$$
which implies that
$J(u)=\widetilde{m}(c)$ and $u\in \mathcal{P}_c.$ Therefore, using Lemma \ref{lm3.4} and Lemma \ref{sec-3-JA>}, it is easy to see that there exists some $\lambda>0$ such that
$(u, \lambda)$ solves \eqref{problem} with  \eqref{eq:20210816-constraint}.

We complete the proof.
\end{proof}


\vskip 0.1cm
\vspace{1cm}
\begin{thebibliography}{99}
\bibitem{ac} C. O. Alves,  F. J. S. A. Crrea. {\it On existence of solutions for a class of problem involving a nonlinear operator}. Comm. Appl. Nonlinear Anal. \textbf{8}(2001), 43--56.

\bibitem{DPS} Y. B. Deng, S. J. Peng, W. Shuai. {\it Existence and asymptotic behavior of nodal solutions for the
Kirchhoff-type problems in $\R^3$}. J. Funct. Anal. {\bf269}(2015), 3500--3527.

\bibitem{ZD} Y. H. Ding, X. X. Zhong. {\it Normalized solution to the Sch\"{o}dinger equation with potential and general nonlinear term: Mass super-critical case}. \newblock {\em arXiv e-prints} (2021). \url{https://arxiv.org/abs/2111.01687}

\bibitem{F} G. M. Figueiredo, N. Ikoma, J. R. Santos Junior. {\it Existence and concentration result for the Kirchhoff type equations with general nonlinearities}. Arch. Ration. Mech. Anal. {\bf 213}(2014), 931--979.

\bibitem{G} Z. J. Guo. {\it Ground states for Kirchhoff equations without compact condition}. J. Differential Equations {\bf259}(7) (2015), 2884--2902.

\bibitem{HLZX} Q. H. He, Z. Y. Lv, Y. M. Zhang, X. X. Zhong. {\it Positive normalized solution to the Kirchhoff equation with general nonlinearities of mass super-critical}. \newblock {\em arXiv e-prints} (2021). \url{https://arxiv.org/abs/2110.12921}

\bibitem{hz12} X. M. He, W. M. Zou. {\it Existence and concentration behavior of positive solutions for a Kirchhoff equation in $\mathbb{R}^3$}. J. Differential Equations \textbf{252}(2012), 1813--1834.

\bibitem{H} Y. He. {\it Concentrating bounded states for a class of singularly perturbed Kirchhoff type equations with a general nonlinearity}. J. Differential Equations {\bf261}(11) (2016), 6178--6220.

\bibitem{HL} Y. He, G. B. Li. {\it Standing waves for a class of Kirchhoff type problems in $\mathbb{R}^3$ involving critical Sobolev exponents}. Calc. Var. {\bf 54}(2015), 3067--3106.

\bibitem{HLP} Y. He, G. B. Li, S. J. Peng. {\it Concentrating bound states for Kirchhoff type problems in $\R^3$ involving critical Sobolev exponents}. Adv. Nonlinear Stud. {\bf14}(2014), 441--468.

\bibitem{Jeanjean1997}L. Jeanjean. {\it Existence of solutions with prescribed norm for semilinear elliptic equations}. Nonlinear Anal. Theory T. M.  A. {\bf28}(1997), 1633--1659.

\bibitem{JeanZhangZhong2021}
L.~Jeanjean, J.~J.~Zhang, X.~X.~Zhong.
\newblock A global branch approach to normalized solutions for the
  {S}chr\"odinger equation. \newblock {\em arXiv e-prints} (2021). \url{https://arxiv.org/abs/2112.05869}

\bibitem{Ki} G. Kirchhoff, Mechanik, Teubner, Leipzig, 1883.

\bibitem{Kwong1989} M. K. Kwong. {\it Uniqueness of positive solutions of $\Delta u-u+u^p=0$ in $\mathbb R^n$}. Arch. Rational Mech. Anal. {\bf 105}(1989), no. 3, 243--266.

\bibitem{13} G. B. Li, P. Luo, S. J. Peng, C. H. Wang, C.-L. Xiang. {\it A singularly perturbed Kirchhoff problem revisited}. J. Differential Equations. {\bf268}(2020), no. 2, 541--589.


\bibitem{Y4} G. B. Li, H. Y. Ye. {\it Existence of positive ground state solutions for the nonlinear Kirchhoff type equations in $\R^3$}. J. Differential Equations {\bf257}(2014), 566--600.

\bibitem{JLLions1978} J. L. Lions. {\it On some questions in boundary value problems of mathematical physics}, in: Contemporary Developments in Continuum Mechanics and Partial Differential Equations, Proceedings of International Symposium, Inst. Mat., Univ. Fed. Rio de Janeino, 1977, in: North-Holl. Math. Stud., vol. 30, North-Hollad, Amsterdam, 1978, pp. 284--346.

\bibitem{14} P. Luo, S. J. Peng, C. H. Wang, C.-L. Xiang. {\it Multi-peak positive solutions to a class of Kirchhoff equations}. Proc. Royal. Soc. Edinburgh. A {\bf149}(2019), no 4, 1097--1122.

\bibitem{lw} X. Luo, Q. F. Wang. {\it Existence and asymptotic behavior of high energy normalized solutions for the Kirchhoff type equations in $\R^3$}. Nonlinear Anal: RWA, \textbf{33}(2017), 19--32.

\bibitem{WT} J. Wang, L. X. Tian, J. X. Xu, F. B. Zhang. {\it Multiplicity and concentration of positive solutions for a Kirchhoff type problem with critical growth}. J. Differential Equations {\bf253}(2012), 2314--2351.


\bibitem{WM} M. I. Weinstein. {\it Nonlinear Schr\"{o}dinger equations and sharp interpolation estimates}. Commun. Math. Phys. \textbf{87}(1983), 567--576.

\bibitem{Wu} X. Wu. {\it Existence of nontrivial solutions and high energy solutions for Schr\"odinger-Kirchhoff-type equations in $\R^N$}. Nonlinear Anal. Real World Appl. {\bf12}(2011), 1278--1287.

\bibitem{ZY-AM} Z. Yang. {\it A new observation for the normalized solution of the Schr\"{o}dinger equation}. Arch. Math. \textbf{115}(3), 329-338, 2020.

\bibitem{Y2} H. Y. Ye. {\it The existence of normalized solutions for $L^2$-critical constrained problems related to Kirchhoff equations}. Z. Angew. Math. Phys. {\bf66}(2015), 1483--1497.

\bibitem{Y3} H. Y. Ye. {\it The mass concentration phenomenon for $L^2$-critical constrained problems related to Kirchhoff equations}. Z. Angew. Math. Phys. {\bf67}(2): 29(2016).

\bibitem{Y1} H. Y. Ye. {\it The sharp existence of constrained minimizers for a class of nonlinear Kirchhoff equations}. Math. Methods Appl. Sci. {\bf38}(2015), 2663--2679.

\bibitem{ZZZZ2021}
X. Y. Zeng, J. J. Zhang, Y. M. Zhang, X. X. Zhong. {\it Positive normalized solution to the Kirchhoff equation with general nonlinearities}. \newblock {\em arXiv e-prints} (2021). \url{https://arxiv.org/abs/2112.10293}

\bibitem{zx} X. Y. Zeng, Y. M. Zhang. {\it Existence and uniqueness of normalized solutions for the Kirchhoff equation}. Applied Math. Letters. \textbf{74}(2017), 52--59.

\bibitem{zc} Y. L. Zeng, K. S. Chen. {\it Remarks on normalized solutions for $L^2$-critical Kirchhoff problems}. Taiwanese Journal of Mathematics, {\bf20}(2016), no. 3, 617--627.
\end{thebibliography}
\end{document}